\documentclass[11pt]{amsart}
\usepackage{latexsym,amssymb,amsmath,amsthm}
\usepackage[mathscr]{eucal}
\numberwithin{equation}{section}

\def\sI{\mathscr{I}}
\def\sJ{\mathscr{J}}
\def\cS{\mathcal{S}}\def\cH{\mathcal{H}}
\def\bC{\mathbb{C}}

\def\bP{\mathbb{P}}

\def\fm{\mathfrak{m}}

\def\tDet{\mathrm{Det}}

\def\tgeneral{\mathrm{general}}
\def\tGL{\mathrm{GL}}
\def\tHilb{\mathit{Hilb}}
\def\tId{\mathrm{Id}}
\def\tker{\mathit{ker}}
\def\bm{\mathbf{m}}

\def\tPrm{\mathrm{Perm}}

\def\tmult{\mathrm{mult}}

\def\tZeros{\mathit{Zeros}}
\def\half{\tfrac12}

\def\td{\mathrm{d}}
\def\w{\omega}
\def\N{{\hbox{\tiny{\textsf{N}}}}}
\def\uN{\underline{e}_\N}
\def\ule{\underline{e}}
\def\lefthook{\hbox{\small{$\lrcorner\, $}}}

\def\tsing{\mathrm{sing}}

\def\be{\begin{equation}}
\def\ene{\end{equation}}

\def\intprod{\negthinspace
\mathbin{\raisebox{.4ex}{\hbox{\vrule height .5pt width 5pt depth 0pt %
        \vrule height 3pt width .5pt depth 0pt}}}}

\def\t{\tau}

\def\frp#1{\frac{\partial}{\partial{#1}}}

\def\ooo#1#2{\omega^{#1}_{#2}}

\def\oo#1{\omega^{#1}_0}

\def\cf{\mathcal F}

\def\BC{\mathbb C}

\def\BP{\mathbb P}
\def\pp#1{\mathbb P^{#1}}

\def\fgl{\mathfrak g\mathfrak l}

\def\pp#1{{\mathbb P}^{#1}}
\def\tdim{\mathrm{dim}}
\def\hd{,...,}

\def\upperp{{}^\perp}

\def\cU{\mathcal U}

\def\cF{{\mathcal F}}
\def\cM{{\mathcal M}}

\def\cS{{\mathcal S}}
\def\cC{{\mathcal C}}

\def\11{\mathbf 1}

\def\FS{{\mathfrak S}}

\def\a{\alpha}

\def\o{\omega}
\def\upperp{{}^\perp}

\def\b{\beta}

\def\d{\delta}

\def\ot{{\mathord{\,\otimes }\,}}
\def\op{{\mathord{\,\oplus }\,}}

\def\ra{{\mathord{\;\rightarrow\;}}}

\def\La#1{\Lambda^{#1}} 

\def\tdeg{\mathit{deg} }

\def\tdim{\text{dim}\,}

\def\tker{\text{ker}\,}

\def\trank{\mathrm{rank}\,}

\def\mpg{\cM_d^{p ,G}}\def\mp{\cM_d^{p}}
\def\xpg{\bold\Xi^{p ,G}_d}\def\xp {\bold\Xi^{p }_d}
\def\ppg{\bold\Phi^{p ,G}_d}\def\ppx{\bold\Phi^{p }_d}

\newtheorem*{prop*}{Proposition}

\newtheorem{theorem}{Theorem}[section]
\newtheorem{proposition}[theorem]{Proposition}
\newtheorem{lemma}[theorem]{Lemma}

\theoremstyle{definition}
\newtheorem{definition}[theorem]{Definition}

\theoremstyle{remark}
\newtheorem{remark}[theorem]{Remark}

\newcounter{Lcnt}

\begin{document}

\address{Department of Mathematics\\Texas A\&M University\\
College Station, TX  77843}
\email{jml@math.tamu.edu, robles@math.tamu.edu}

\title[Lines and osculating lines of hypersurfaces]{Lines and osculating lines of hypersurfaces}
\author{J.M. Landsberg \and C. Robles}
\date{\today}
\begin{abstract}This is a detailed study of the infinitesimal
variation of the varieties of lines and osculating lines through 
a point of a low degree hypersurface in projective space.   The motion is governed by a system of partial differential
equations which we describe explicitly.
\end{abstract}

\maketitle

\section{Introduction}

\subsection{Motivation and context}

Let   $X^n\subset  \BC\pp{n+1}$ be a  hypersurface with $\tdeg(X) \leq n$.  Given a   point $x\in X$, let $ \cC_x\subset \BP T_xX$ denote the tangent directions to lines $\ell = \bP^1$ such that $x \in \ell \subset X$.  If $\tdeg(X) \leq 2$ and $x,y\in X$ are general points, then $\cC_x$ is projectively isomorphic to $\cC_y$.   
The goal of this paper is to answer a question posed by Jun-Muk Hwang:
\begin{quote}
{ Let $\tdeg(X) \geq 3$. 
  How can $\cC_x$ vary (modulo projective isomorphism) as 
  $x$ varies over the general points of $X$?}
\end{quote}
The question is motivated by Hwang and N. Mok's program to 
study Fano varieties  via  the variety of tangent directions to minimal degree rational
curves through a   point.  (See \cite{Hwang_Trieste} for an overview.)

The first interesting case is $d=3$ and $n=4$.  Here $\cC_x$ is a curve of degree six
in $\pp 3$ -- i.e., a genus four curve in its canonical embedding.
When $d=3$ and $n=5$, $\cC$ is a $K3$ surface in $\pp 4$.  

In another direction, this paper is a continuation of a program to understand the relationship
between the algebraic geometry of subvarieties of projective space and their local projective
differential geometry.
This was a project of the classical   geometers (e.g. \cite{Fubini1916,Severi,MR1190006}), was revived by Griffiths and Harris in \cite{GH} 
and continued in \cite{JenMus,Lline,MR2011327,Lsnu,Lrigid,R_SL,LRfub,MR2136366,MR2014407,MR1876644}, for example. 

\smallskip

Before addressing Hwang's question, it will be useful to  generalize it as follows: 
Fix $d\leq n$ and
let $X^n\subset\pp{n+1}$ be a variety of degree at least $d$.  
Given $x\in X $, let 
$\cC_{d,x}\subset \BP T_xX$ denote the {\it tangent directions to lines osculating
to order $d$ with $X$ at $x$}. (See \S\ref{sec:frames} for a precise definition and a discussion  of osculation).  If $d=\tdeg(X)$, then $\cC_{d,x} = \cC_x$. 
For sufficiently general $X$ of degree at least $d$ and general $x\in X$, $\cC_{d,x}\subset \BP T_xX$ is
the transverse intersection of smooth  hypersurfaces   of degrees $2\hd d$. In particular it has codimension $d-1$ in $\BP T_xX$.   

\subsection{The results}

Our results are technical and require substantial notation to state precisely.  
  Here we present central ideas and roughly state the results.  
The terminology   will be made precise in Sections \ref{sec:frames} and \ref{sec:frame_thms}.

Let $X\subset \pp{n+1}$ be a hypersurface of degree at least $d$.  
We say $x\in X$ is a \emph{$d$-general point} if    
all discrete projective differential invariants up to order $d$ are constant
on a Zariski open subset of $X$ containing $x$. Let $p$ denote the Hilbert polynomial
of $\cC_{d,x}\subset \BP T_xX$. 
Let $U\subset X$ be a neighborhood of $x$ that admits a local framing 
(cf. \S\ref{sec:F1}) consisting of  $d$-general points. 
A choice of a local framing $e$ yields, for all $y\in U$, polynomials
$(F_{2,e,y}\hd F_{d,e,y})$, with $F_{\d,e,y}\in S^{\d}T^*_yX$.  While the polynomials depend on our choice of frame, their zero set in $\BP T_yX$ does not; that variety is $\cC_{d,y} \subset \bP T_yX$.  The framing enables us to
identify the tangent spaces at different points with a fixed complex vector space $T$ of dimension $n$.  Then we may regard these polynomials as elements of $S^\d T^*$.  This determines a map
$$
  \hat\phi_{d,e}: U\ra \tHilb_{p } \, , 
$$ 
where $\tHilb_p$ is the Hilbert scheme of
subvarieties of $\BP T=\pp{n-1}$ having Hilbert polynomial $p$.

In order to eliminate the ambiguity in our choice of framing (which determines the identification of $T_xX$ with $T$) we wish to quotient the image by the action of $GL(T)$.  The quotient of $\tHilb_p$ by $GL(T)$ is not a manifold or algebraic variety. However, the quotient
of $\hat\phi_{d,e}(U)\subset \tHilb_p$ will usually be a manifold. For example, if $X$ is sufficiently general,
  the stabilizer
in $GL(T)$ of a point in $\hat\phi_{d,e}(U)$ will be trivial. For simplicity of discussion, for
the moment assume this is the case. Let $\tHilb_p^0$ denote the open subset of $\tHilb_p$ where
the stabilizer in $GL(T)$ is trivial and such that $\tHilb_p^0$ contains $\hat\phi_{d,e}(U)$
and is a manifold.  Set 
$$ 
  \cM^p_d \ = \ \tHilb_p^0 /  GL(T) \, .
$$
which by definition is a manifold.
We obtain a well defined map
$$
  \phi_{U,d} : U\ra \mp  \, .
$$
This map is independent of choice of local framing, so it extends to a well-defined map
\be\label{phiddef}
\phi_d : X_{\tgeneral}\ra\mp\, .
\ene

For most $X$ the Hilbert polynomial $p$ will be that of a generic complete intersection of type $(2\hd d)$ in $\pp{n-1}$, and the stabilizer $G$ will be trivial.  (The Hilbert polynomial is that of a generic complete
intersection if and only if the polynomials $F_{2,e,x}\hd F_{d,e,x}$ have no non-trivial syzygies.)    We denote this generic moduli space by $\cM_d$.  Describing $\td \phi_d$
in the case that $d$ is the degree of $X$ answers Hwang's question. 
We show 
\begin{list}{}{ \usecounter{Lcnt}
 \setlength{\leftmargin}{40pt} \setlength{\itemindent}{25pt}
 \setlength{\labelwidth}{40pt} }
%
%
\item[(Theorem \ref{phixthmb}.a)]  The image
of $\phi_d$  satisfies a first-order system of
partial differential equations. The system is expressed by the condition that the Gauss map $\gamma_{\phi_d}$ associated to $\phi_d$ takes image in a proper
subvariety $ \xp $ of the Grassmann bundle of $\mp $.  We give an explicit description of $ \xp $ in \S\ref{sec:quotient}. 
\item[(Theorem \ref{phixthmb}.b)]  When $\tdeg (X)=d$ (so that $\cC_{x,d}=\cC_x$) the image of $\phi_d$
satisfies a more restrictive first-order system of partial differential
equations: the Gauss image must lie in a smaller subvariety
$\ppx\subset \xp $. We give an explicit description of $\ppx$ in \S\ref{sec:quotient}. 
\item[(Theorem \ref{thm:deg=d})]  Suppose that $\tdeg(X) \geq d$ and that the Gauss image
of $\phi_d(X_{\tgeneral})$ lies in $\ppx$.  Then, under mild genericity conditions on $X$, we may conclude that $\tdeg(X)=d$. (More precisely,  it is sufficent to assume that either $X$ is smooth or that
for a general line $\ell\subset X$, $\ell\cap X_{\tsing}=\emptyset$.)
\item[(Theorem \ref{forcedmove})]    There is a non-empty
Zariski open subset $A_d $ of $\cM_d $, such that for any $X$ with 
$\tdeg(X)\geq d\geq 4$ and $x\in X_{\tgeneral}$ with $\phi(x)\in A_d $,
$\trank(\td \phi_d|_x)=n$.
\end{list}
Theorem \ref{forcedmove} is suprising as it gives conditions that imply
that $\cC_x$ \lq\lq must move as much as possible\rq\rq\ without any hypotheses on $X$.

In the special case that $d=3$, Theorem \ref{phixthmb}.b is related to a classical result of Fubini and
Cartan.  Fubini
\cite{Fubini1916} stated (and Cartan \cite{Cartan_projDef} gave a rigorous proof) that
under the hypotheses of the theorem, $X$ is projectively determined by $\phi_3(U)$, $U\subset X$ any open subset in the analytic topology.
For a short, modern proof, see   \cite{JenMus}. 

\begin{remark} In \cite{JenMus}    they also extend the Cartan-Fubini theorem when $n\geq 3$
to the case that the second fundamental form is degenerate.
The result fails in the case of surfaces, see \cite{Cartan_projDef}, though it still
holds \lq\lq generically.\rq\rq
\end{remark}

\medskip

The results above are consequences of Theorems \ref{phixthm} and \ref{fullrank1} on  the Fubini forms (cf. \S\ref{sec:fubini}) of the hypersurface $X$.
More precisely the $\cC_{d,x}$ do not contain all the geometric information of
$X$ accessible by $d$ derivatives at a point. This information  is contained
in the Fubini forms $(F_{2,e,x}\hd F_{d,e,x})$
 -- cf.  \S\ref{sec:frame_thms}, with $F_{\d,e,x} \in S^\d T^*$.
  (The exception is the case $d=3$, where $\cC_x$ contains all the information of the Fubini forms.)  The collection of Fubini forms are not well-defined
at $x \in X$, but do define an equivalence class under the action of
a group $H$ of dimension $n^2+2n+3$, see \S\ref{sec:sharper}. 
 
We will first establish results in a fixed   local framing $e$ which fixes a choice of Fubini forms.    In particular, we have a map
\be\label{tildephidedef}
  \tilde\phi_{d,e} : U \ \ra \ \oplus_{\d=2}^dS^{\d}T^* \, .
\ene
We determine a first-order system of partial differential equations that $\tilde\phi_{d,e}$ must satisfy (Theorem \ref{phixthm}).  We   quotient to the Hilbert scheme and then by the action of $GL(T)$ to obtain the choice-free results listed above.  This is carried out in Section \ref{sec:quotient}.

Several examples of cubic hypersurfaces are considered in \S\ref{sec:examples}.

\section{Local frames and Fubini forms}\label{sec:frames}

\subsection{Notational conventions}
For subsets 
$X\subset \BP V$, $\widehat X\subset V$ denotes
the corresponding cone.  For a submanifold $X\subset \BP V$ and $x \in X$,
$\widehat T_xX   \subset V$ denotes its \emph{affine tangent
space}.  The \emph{tangent space} and \emph{normal space} at $x$ are  
$$
  T_xX \ = \ {\hat x}^* \ot (\widehat T_{x}  X / {\hat x})
  \quad \hbox{ and } \quad 
  N_x X = T_x\bP V/T_xX = {\hat x}^* \ot (V/\widehat T_{x} X) \, ,
$$
respectively.

We will use the following index ranges:
\begin{align*}
&0\leq j,k,\ell\leq n+1 \, , \\
&1\leq a,b,a_j\leq n \, , \\
&\textsf{N}=n+1 \, .
\end{align*}
The linear span of vectors $\{v_1\hd v_k\}$ is denoted $\langle v_1\hd v_k \rangle$.

\subsection{Adapted frames}\label{sec:F1}
Let $\pi : \BC^{n+2}\backslash 0\ra \pp{n+1}$ denote the natural projection $v \mapsto [v]$. 
Let 
$X^n\subset\pp{n+1}$ be a submanifold and let $\cF^1\ra X$ denote
the bundle of first-order adapted frames.  Elements of $\cF^1 \subset \tGL(V)$ are frames (or bases) $e = (e_0,e_1,\ldots,e_\N)$ of $V=\BC^{n+2}$, such that 
\begin{eqnarray*}
  \pi(e_0) & \in & X \, , \\
  \widehat T_{[e_0]}  X & = & \langle e_0 , \ldots , e_n \rangle \, .
\end{eqnarray*}
A \emph{local, first-order adapted framing} is a section $e : U \to \cF^1$, $U \subset X$ open in the analytic topology.  Every smooth point on a projective variety admits a neighborhood $U$ with local framing (\S\ref{sec:coords}).

Since $X$ is a hypersurface, $N_xX$ is a line bundle spanned by 
\begin{displaymath}
  \uN := e^0 \ot ( e_\N \hbox{ mod }  \widehat T_{e_0}   X ) \, .
\end{displaymath}
Here  $(e^0,\ldots, e^\N)$ is the basis of $V^*$ dual to $e$.

Define the $\fgl(V)$-valued Maurer-Cartan form $\w = \w^j_k e_j\ot e^k$ on
$GL(V)$ by $\td e_j = \w^k_j \, e_k$. Recall the Maurer-Cartan equation:
$$
  \td \, \w^j_k \ = \ -\w^j_\ell \wedge \w^\ell_k \, .
$$
We abuse notation and denote the pullback
of $\w$ to $\cf^1$ by $\o$ as well.    
\subsection{Fubini forms}\label{sec:fubini}
The variety $\cC_{k,x}$ will be defined (Definition \ref{dfn:C}) as the zero set of the Fubini forms $(F_{2,e,x}\hd F_{k,e,x})$.  We review Fubini forms here; see \cite[Ch. 3]{IL} for details.

Since $\widehat T_{[e_0]}  X = \langle e_0 , e_1 , \ldots , e_{n } \rangle$ we have $\w^\N_0 = 0$ on $\cf^1$.  Differentiating this equation, and an application of the Cartan lemma (see, e.g.  \cite[p. 314]{IL}), yields
functions $r_{ab} = r_{ba} : \cF^1 \to \bC$ such that
\begin{equation}\label{eqn:F2}
  \w^{\N}_a \ = \ r_{ab} \, \w^b_0 \ , \qquad 
  1 \le a , b \le n  \, .
\end{equation}
The coefficients $r_{ab}$ define the Fubini quadric (also known as
the projective second fundamental form) $F_2 = r_{ab} \,\ule^a \ule^b \ot \uN\in \Gamma(\cF^1,\pi^*(S^2T^* X \ot N X))$.  Here the $\ule^a \in T^*_x X$ are dual to the basis 
\begin{displaymath}
  \ule_a = e^0 \ot ( e_a \ \mathrm{mod} \ e_0 ) 
\end{displaymath}
of $T_xX$ and $x=\pi(e_0)$.

The coefficients of the Fubini cubic are obtained by differentiating \eqref{eqn:F2} and another application of Cartan's Lemma.  The coefficients $r_{a_1a_2\cdots a_p}$ of the $p$-th Fubini form 
$
  F_p = r_{a_1\cdots a_p} \, \ule^{a_1} \cdots \ule^{a_p} \ot \uN \in \Gamma(\cF^1,\pi^*(S^pT^* X \ot N X))
$
are defined inductively.  The defining formula is as follows.
Let $\FS_{p+q}$ denote the symmetric group on $p+q$ letters.  Given two tensors $T_{a_1 \ldots a_p}$ and $U_{a_{p+1} \ldots a_{p+q}}$,   let
$$
  T_{(a_1 \ldots a_p} U_{a_{p+1} \ldots a_{p+q})} \ = \ 
  \frac{1}{(p+q)!} \sum_{\sigma\in \FS_{p+q}} \, 
  T_{\sigma(a_1) \ldots \sigma(a_p)} U_{\sigma(a_{p+1}) \ldots \sigma(a_{p+q})}
$$ 
denote the {\it symmetrization} of their product.  For example, $T_{(a_1} U_{a_2)} = \half( T_{a_1} U_{a_2} + T_{a_2} U_{a_1})$.  We exclude from the symmetrization operation any index that is outside the parethenses.  For example, in $r_{b(a_1 \ldots a_{p-1}} \, \w^b_{a_p)}$ we symmetrize over only the $a_i$, excluding the $b$ index.
Define $$ r_a = 0 \, . $$
\begin{prop*}[\cite{R_SL}]
\label{prop:F}
The coefficients of $F_{p+1}$, $p>1$,   fully
symmetric in their lower indices, are defined by
\begin{eqnarray*}
  r_{a_1 \ldots a_p b} \, \w^b_0 & = &
  - \, \td r_{a_1 \ldots a_p} \, - \, r_{a_1 \ldots a_p} \,
  \left\{ (p-1) \,  \w^0_0 \, + \, \w^\N_\N \right\} \hfill \\
  & &
  + \ p \,  
  \left\{ (p-2) \, r_{(a_1 \ldots a_{p-1}} \w^0_{a_p)} \,  
       + \, r_{b(a_1 \ldots a_{p-1}} \, \w^b_{a_p)}  \right\} \\
  & & 
  - \ \sum_{j=1}^{p-2} \, \binom{p}{j} \, 
      \bigg\{ 
      \ r_{b(a_1 \ldots a_j} \, r_{a_{j+1} \ldots a_p)} \, \w^b_\N \ 
      + \ (j-1)\, r_{(a_1 \ldots a_j} \, r_{a_{j+1} \ldots a_p)} \, 
                 \w^0_\N \, \bigg\} \, . 
\end{eqnarray*}
\end{prop*}

The Fubini forms $F_p$ are defined on $\cF^1$ and do not descend to well-defined sections of $S^pT^*X\ot NX$ over $X$.  However, the ideal generated by $\{ F_2 , F_3 , \ldots , F_k \}$ is independent of our choice of first-order adapted frame over $x\in X$.  

\begin{proposition}\label{easyframe} Given $X^n\subset\pp{n+1}$ and a smooth point $x\in X$, there exists a local framing about $x$ such that all entries  of the Maurer-Cartan form (pulled back via this framing) vanish with the exception of $\oo a$ and $\ooo Na$, $1\leq a\leq n$.
\end{proposition}

\begin{proof} One such local framing is given in \S\ref{sec:coords} below.
\end{proof}

\begin{definition} \label{dfn:C}
Given a smooth point $x \in X$, $\cC_{k,x} \subset \bP T_xX$ is the zero locus of $F_{2,x} , \ldots , F_{k,x}$.
\end{definition}

\subsection{Fubini forms in a coordinate framing}\label{sec:coords}
Given a smooth point $x \in X$, fix homogeneous coordinates $[z_0:z_1:\cdots:z_\N]$ of $\bP V$ so that $x = [1:0:\cdots:0]$, and $\widehat X$ is tangent to $\{ z_\N=0 \}$ at $(1,0,\ldots,0) \in \hat x \subset V$.  Setting $z_0=1$ yields a coordinate neighborhood on $\bP^\N$ with local coordinates $(z^1,\ldots,z^\N)$ centered at $x$.  Shrinking the coordinate neighborhood $U$ if necessary, we may assume that $U \cap X$ is a graph $z^\N = f(z^1,\ldots,z^n)$ over its embedded tangent space at $x$.    A local, first-order adapted frame $e : U \to \cF^1$ over $U$ is defined by 
\begin{eqnarray*}
  e_0 & = & \frac{\partial}{\partial z^0} + z^a \frac{\partial}{\partial z^a}
          + f \frac{\partial}{\partial z^\N} \\  
  e_a & = & \frp{z^a} + f_{z^a} \frp{z^\N}  \\
  e_\N & = & \frp{z^\N} \, ;
\end{eqnarray*}  
here, $f_{z^a}$ denotes partial differentiation, and $1 \le a \le n$.  With respect to this frame, the only nonzero entries in the Maurer-Cartan form $\w = \w^j_k e_j \ot e^k$ are 
\begin{displaymath}
  \w^a_0 \ = \ \td z^a \qquad \mathrm{and} \qquad 
  \w^\N_a \ = \ f_{z^az^b} \, \td z^b \, .
\end{displaymath}
Regard $z = (z^1,\ldots,z^n)$ as local coordinates on $X$.  It is immediate from the expression for $\w^\N_a$ above that the coefficients of the second Fubini form at $z$ are $r_{ab}(z) = f_{z^az^b}(z)$.  More generally, Proposition \ref{prop:F} implies the Fubini forms at $z$ are given by 
\begin{displaymath}
  F_{k,e}(z) = (-1)^k f_{z^{a_1}\cdots z^{a_k}}(z) \, 
        \td z^{a_1} \cdots \td z^{a_k} \ot 
        \underline{e_\N} \, .
\end{displaymath}
See \cite[\S 3.3.7]{IL} for details.
\medskip


If the $y = (y^1 , \ldots , y^n )$ are linear coordinates on $T_zX$ induced by $\{ e_1 , \ldots , e_n \}$, it follows that $\cC_{k,z}$ is the zero set of the polynomials 
$$
f_{z^{a_1}\cdots z^{a_\ell}}(z) \, y^{a_1} \cdots y^{a_\ell} \, , \quad  2\leq \ell\leq k  \, .
$$
In particular, 
$$
\cC_{k,x}=\{\ell\in \BP T_xX\mid \exists
L=\pp 1\subset \pp n,\ \BP T_xL=\ell,\ \tmult_x(X\cap L)\geq k+1\}.
$$

\begin{lemma}  \label{lem:choice}
  Given $\cC\subset \BP T$, a complete intersection of hypersurfaces of
degrees $2\hd d$ and $d' \ge d$, there exists a hypersurface $X$ of degree $d'$, a point $x \in X$,
and a local framing $e$  such that $\tZeros(\tilde\phi_{d,e}(x)) = \cC$.
\end{lemma}

\noindent{\it Remark.} Thus given
$\bold P\in \oplus_{\d=2}^dS^{\d}T^*$ there exist  $X$ and $x\in X$ such that $\phi_d(x)=\bold P$.   It is not clear which points of $\oplus_{\d=2}^dS^{\d}T^*$ can be the image of a \emph{general} point of some $X$.

\begin{proof}
  Pick $\bold P=(P_2\hd P_d)$, with $P_j\in S^jT^*$ so that $\cC = [\tZeros(\bold P)]$.  Let $f(z) = f(z^1,\ldots,z^n)$ be a polynomial of degree $d'$ and let $f_k(z)$ denote the degree $k$ homogeneous component of $f(z)$.  Specify $f_0 = 0$, $f_1 = 0$ and $f_\d = P_\d$ for $2 \le \d \le d$.  Take $X$ to be the closure of the graph of $f$, and $x = [1:0:\cdots:0]$.
\end{proof}

\section{Results in terms of Fubini forms}\label{sec:frame_thms}

\subsection{Gauss maps}\label{sec:gauss}

We will consider two types of Gauss map: the Gauss map of a hypersurface in $\BP V$, and the Gauss map associated to a differentiable map between manifolds.  Given a vector space $V$, let $G(m,V)$ denote the Grassmannian of $m$-planes through the origin in $V$.  

The \emph{Gauss map $\gamma_X$ of a hypersurface} $X \subset \bP^\N = \bP V$ is
\begin{eqnarray*}
 \gamma_X: X_\mathrm{smooth} & \ra & G(n+1,V) \, ,\\
 x & \mapsto & \gamma_X(x) := \widehat T_x X \, .
\end{eqnarray*}
The tangent space $T_E G(n+1,V)$ may be identified with $(V/E) \ot E^*$.  Making use of a frame $e$ we may identify $T_{\gamma_X(x)} G(n+1,V)$ with $N_xX\ot\hat x  \ot (\widehat T_x X)^*$.  Under the identification the differential is given by $\td\gamma_{X,x}(v^a \underline e_a) = v^a r_{ab} \, \underline e_\N \ot e^b \in  N_xX\ot \hat x \ot (\widehat T_x X)^*$.  

The Gauss map $\gamma_X$ is \emph{nondegenerate} if   $\td\gamma_{X,x}: T_xX\ra T_{\gamma_X(x)} G(n+1,V)$ is of maximial rank $n$.  The Gauss map is nondegenerate if and only if the quadratic polynomial $F_{2,e}$ is of full rank; equivalently, the quadric hypersurface $\{ F_{2,e}(x)=0 \} \subset \bP T_x X$ is smooth.

\medskip

\noindent{\it Remark.}  The second Fubini form $F_2$ and (projective) second fundamental form (defined by $\td \gamma_X$) agree.  See \cite[Ch. 3]{IL}.

\medskip

Given a manifold $\Sigma$,
let $\bold G(n,T\Sigma) \to \Sigma$ denote the Grassmann bundle, whose fiber over $\sigma\in \Sigma$ is $G(n,T_\sigma\Sigma)$.
Let $f: Z\ra \Sigma$ be a $C^1$ map of manifolds of generic rank $r$ and set $Z' = \{z\in Z\mid \trank(df_z)=r \} \subset Z$.  The \emph{Gauss map $\gamma_f$ associated to $f$} is
\begin{align*}
\gamma_f: Z'&\ra  \bold G(r,T\Sigma) \\
  x &\mapsto 
  T_{f(x)}f(Z).
\end{align*} 

\subsection{Results}
Given $S=(S_3\hd S_{d+1})\in \oplus_{\d=3}^{d+1}S^{\d}T^*$, such that the common zero locus of
the $S_j$ is not a cone, we obtain
an $n$-plane $E_S\in G(n,\oplus_{\d=2}^dS^{\d}T^*)$ by
$$
E_S=\langle (v\lefthook S_3\hd v\lefthook S_{d+1})\mid v\in T\rangle \, .
$$
Here $v\lefthook: S^{\d}T^*\ra S^{\d-1}T^*$ denotes the interior product, i.e.,
$(v\lefthook P)(w_1\hd w_{\d-1})=P(v,w_1\hd w_{\d-1})$.
Fix $\bold P= (P_2\hd P_d)\in \oplus_{\d=2}^{d }S^{\d}T^*$.  
Define a map
\begin{eqnarray*}
  \mu_{\bold  P} :  S^{d+1}T^* & \ra & \oplus_{\d=3}^{d+1}S^{\d}T^* \, , \\
  \alpha & \mapsto &  (P_3\hd P_d, \alpha) \, .
\end{eqnarray*}
Let $I_\d(\bold P)$ denote the degree $\d$ homogeneous component of the ideal generated by $(P_2\hd P_d)$.
\begin{definition}\label{xiphidef} Using the notations  $E_S,\mu_{\bold  P}$ above, 
define the varieties
\begin{eqnarray*}
\bold\Xi_{\bold P} & := & \{ E_{\mu_{\bold P}(\a)} \mid  \a\in S^{d+1}T^* \}
  \ \subset \ G(n, \oplus_{\d=2}^dS^{\d}T^*) \\
\bold\Phi_{\bold P} & := & \{ E_{\mu_{\bold P}(\b)}\mid \b\in I_{d+1}(\bold P) \}
  \ \subset \ G(n, \oplus_{\d=2}^dS^{\d}T^*) \, .
\end{eqnarray*}
Since $\oplus_{\d=2}^dS^{\d}T^*$ is a vector space, we may identify the tangent space $T_{\bold P}(\oplus_{\d=2}^dS^{\d}T^*)$
with $ \oplus_{\d=2}^dS^{\d}T^*$, and define sub-bundles of the Grassmann bundle:
$\bold \Phi_d\subset \bold \Xi_d\subset\bold G(n, T(\oplus_{\d=2}^dS^{\d}T^*))$ by
$(\bold\Xi_d)_{\bold P}= \bold\Xi_{\bold P}$,  and $(\bold\Phi_d)_{\bold P}= \bold\Phi_{\bold P}$.
\end{definition}
 
\begin{theorem}\label{phixthm}
Let $X^n\subset\pp{n+1}$ be a hypersurface of degree at least $d$, with
nondegenerate Gauss map. Let $U\subset X$ be
a $d$-general open subset admitting a local framing $e$ as in \S\ref{sec:coords}, and consider the map
\begin{equation}\label{tildephimap}
\renewcommand{\arraystretch}{1.4}
\begin{array}{rcl}
\tilde\phi_{d,e}: U  & \ra & \oplus_{\d=2}^dS^{\d}T^*\\
x & \mapsto & \left( F_{2,e}(x)\hd F_{d,e}(x) \right) \, .
\end{array}\end{equation}
Assume further that   $\gamma_{\tilde\phi_{d,e}}$ is defined on $U$; that is,
$\trank(d\tilde\phi_{d,e}|_x)=n$ for all $x\in U$. Then,
recalling the varieties $\bold\Xi_d$ and $\bold\Phi_d$ of Definition \ref{xiphidef},
we have the following:
\begin{list}{\emph{(\alph{Lcnt})}}
{ \usecounter{Lcnt} \setlength{\leftmargin}{20pt}}
\item The Gauss image $\gamma_{\tilde\phi_{d,e}}(U)$ is contained in ${\bold\Xi}_d$.
\item If $\tdeg(X)=d$, then 
$\gamma_{\tilde\phi_{d,e}}(U)\subset \bold\Phi_d$.
\end{list}
\end{theorem}

\noindent{\it Remark.}  The assumption on the rank of $\td \tilde \phi_{d,e}$ is generic; see Proposition \ref{fullrank1} below.


\begin{proof}
Choose a local framing as in Proposition \ref{easyframe}, then
\begin{eqnarray*}
  \td \tilde \phi_{d,e}(z) & = & ( \td F_{2,e}(z) , \ldots , \td F_{d,e}(z) ) \\
  & = & - ( F_{3,e}(z) , \ldots , F_{d+1,e}(z) ) \, .
\end{eqnarray*}
Given $z \in U$, $\gamma_{\tilde \phi_{d,e}}(z)$ is the $n$-plane  
\begin{equation}\label{eqn:dtildephi}
\left\langle  \,  - ( v \lefthook F_{3,e}(z) , \ldots , v \lefthook F_{d+1,e}(z) ) \, \mid v\in T \, \right\rangle .
\end{equation}
  This proves Part (a).

If $\tdeg(X)=d$, then $F_{d+1,e}(x)\in I_{d+1}(\cC_{d,x})$, establishing Part (b).
\end{proof}

%
\begin{proposition} \label{fullrank1}
 Let $X^n\subset\pp{n+1}$ be a general  hypersurface of degree $d' \geq d$  with
nondegenerate Gauss map. Let $U\subset X$ be
a $d$-general open subset admitting a local first-order adapted framing $e$ of $X$ as in Proposition \ref{easyframe}, and consider the map \eqref{tildephimap}.  Then $\trank(\td\tilde\phi_{d,e}|_x)=n$ for all $x\in U$.
\end{proposition} 

\begin{proof}
By lower semi-continuity it suffices to show that $\trank(\td\tilde\phi_{d,e}|_x)=n$ for just one $X$.
A family of such $X$ is given in \S\ref{ranknexample}.
\end{proof}

\subsection{Example}\label{ranknexample}

Fix linear coordinates $\overline z = (z^0,\ldots,z^\N) \in \bC^{\N+1}$.  Set $z = (z^1,\ldots,z^n)$ and let $p_j(z)$ be a homogeneous polynomial of degree $j = 2\hd d$ with $d>2$.  Consider the degree $d$ hypersurface
$X_F$ given by the zero set of:
$$
  F(\overline z) \ = \ - (z^0)^{d-1} \, z^\N + (z^0)^{d-2} \, p_2(z) + (z^0)^{d-3} \, p_3(z) \, +\cdots + p_d(z).
$$
Note that 
$$
  (X_F)_{\tsing}=\{ z^0=0 \}\cap \tZeros(p_{d-1})\cap \tZeros(p_d)_{\tsing} \, .
$$
In particular, $X_F$ will usually be smooth.
The point $x = [1:0:\cdots:0] \in \bP^\N$ lies on   $X_F$.  In the 
affine coordinate neighborhood $U = \{ z^0 = 1 \} \subset \bP^\N$, $X$ may be expressed as a graph 
$$
  z^\N = f(z) =  p_2(z) +\cdots + p_d(z) \, .
$$
Let $e = (e_0,\ldots,e_\N)$ be a first-order framing, constructed as in \S\ref{sec:coords}.  

Let $y = (y^1 , \ldots , y^n )$ be linear coordinates on $T^*$ induced by the frame $e$.  Then 
$$
  F_{\d,e}(z) \ = \ (-1)^\d \, \left(  
  \frac{\partial^\d f}{\partial z^{a_1} \cdots \partial x^{a_\d} } (z)\right) \, 
  y^{a_1} \cdots y^{a_\d} \, .
$$
If $p_{\d}(z) = p_{a_1\hd a_{\d}} \, z^{a_1} \cdots  z^{a_{\d}}$, with $p_{a_1\hd a_{\d}}$ symmetric, then the second Fubini quadric is given by $F_{2,e}(z) = r_{ab}(z) y^a y^b$ with 
$$
 r_{ab}(z) \ = \ 2 \, p_{ab} + \sum_{\d=3}^d j(j-1) \, p_{abc_{3}\cdots c_{\d}} z^{c_3}\cdots z^{c_\d} \, .
$$  
It follows that if $p_2$ is nondegenerate, then the Fubini quadric will be nondegenerate in a neighborhood of $z=0$.  

Additionally, note that $F_{\d,e}(0) = (-1)^\d \d ! \, p_{a_1\cdots a_\d} y^{a_1}\cdots y^{a_\d}$, $2\le\d\le d$, and $F_{\d,e} = 0$ for all $\d > d$, so that 
$$\td \tilde\phi_{d,e}(0)(T) \ = \ 
  \left\langle \left( 3\cdot 3! \, p_{ab_1b_2} \, y^{b_1} y^{b_2} \, , \, \cdots , 
  (-1)^{d-1} d \cdot d! \, p_{ab_1\cdots b_{d-1}}y^{b_1}\cdots y^{b_{d-1}} \, , \, 
  0 \right) \mid 1\leq a\leq n \right\rangle  .
$$
In particular, for a generic choice of $p_3$, the differential $\td \tilde\phi_{d,e}(0)$ will have maximal rank $n$.

\subsection{Exterior differential systems interpretation}\label{sec:EDS}
We may rephrase Theorem  \ref{phixthm}  in the language  of exterior differential systems (EDS) as follows. (See \cite[p. 177]{IL} or \cite{BCG3}.)
Let $\Sigma$ be a manifold and let $\pi: \bold G(n,T\Sigma)\ra \Sigma$ denote the
Grassmann bundle.  The Grassmannian $\bold G(n,T\Sigma)$ carries a canonical linear Pfaffian system $(\sI,\sJ)$.  Let $p \in \Sigma$,  $E \in \bold G(n,T_p\Sigma)$, and $E^\perp \subset T_p^*\Sigma$ be the forms vanishing on $E$.  The canonical EDS on $\bold G(n,T\Sigma)$ is generated by the
subspace $\sI \subset T^*\bold G(n,T\Sigma)$, defined fiber-wise by $\sI_{p,E} = \pi^*(E\upperp)$.  The independence condition $\sJ\supset\sI$ is given by $\sJ_{p,E}=\pi^*(T^*_p\Sigma)$.  Integral manifolds of this tautological system are immersed $n$-dimensional submanifolds 
$i: M\ra \bold G(n,T\Sigma)$ such that $i^*(\sI)\equiv 0$ and
$i^*(\La n(\sJ/\sI))$ is non-vanishing.  They are characterized as the Gauss images of immersed $n$-dimensional submanifolds $N$ of $\Sigma$.
That is, there exists $N$ such that $M = \gamma(N)$, where $\gamma$ is the Gauss map $p \mapsto (p,T_pN)\in G(n,T_p\Sigma)$.  Theorems \ref{phixthm}.a and \ref{phixthm}.b are rephrased as 

\begin{quote}
{\it  The image $\tilde\phi_d(X)$ is an integral manifold of the pull-back of the tautological system $(\sI,\sJ)$ on $\bold G(n,T(\op_{\d=2}^dS^{\d}T^*))$ to $\bold \Xi_d$ (respectively, $\bold \Phi_d$).}
\end{quote}

Other results of this paper can similarly be rephrased in the language of EDS.

\section{Descent to $\cM_d$}\label{sec:quotient}
Fix a polynomial $p$ that occurs as a Hilbert polynomial of a 
codimension $d-1$   complete intersection
of hypersurfaces of degrees $2,3\hd d$ in $\pp{n-1}$ such that the degree two
hypersurface  is smooth.  (Recall that at a general point $x$ on a generic $X$, $\cC_x$ with be such a complete intersection.)
  
Let 
$$\cU_{p}=\{ (P_2\hd P_d)\in \oplus_{\d=2}^dS^{\d}T^* \mid
\tHilb(\tZeros(P_2\hd P_d))=p\}
$$
and let 
$$
  \pi_{p}: \cU_{p}\ra \tHilb_{p}
$$ 
denote the projection map to the Hilbert scheme.
Let $X^n\subset \pp{n+1}$ be a hypersurface with a nondegenerate Gauss map (\S\ref{sec:gauss}); equivalently the Fubini quadric is  of maximal rank  at general points.  Let $x\in X_{\tgeneral}$,
and choose a local first-order adapted framing  $e$ on an open set $U \subset X$, containing $x$ as in Proposition
\ref{easyframe}.
Let $p$ denote the Hilbert polynomial 
of $\tZeros(\tilde\phi_{d,e}(x))$. The polynomial $p$ is independent of our choices of $x$ and $e$.
Consider the map
$$
  \hat \phi_{d,e} = \pi_{p}\circ \tilde\phi_{d,e}: U\ra \tHilb_{p }  \, .
$$
The map $ \hat \phi_{d,e}$ depends on our choice of framing.  (Recall that our identification of $T_xX$ with the fixed vector space $T$ is made via the frame $e$.)  To remove this ambiguity, we 
would like to quotient the image in  $\tHilb_{p}$ by the
action of $GL(T)$.  Unfortunately the quotient of 
$\tHilb_{p}$ by the action of $GL(T)$ is not a manifold.  As mentioned in the introduction, if $X$ is sufficiently general, there is a well-defined quotient
of $\tHilb_p^0$. In the case there is a nontrivial stabilizer $G$ of
$\hat\phi_{d,e}(x)\in \tHilb_{p}$,   
let $\cH^G\subset \tHilb_{p }$ denote the connected component
of $\tHilb_{p }$ with
  stabilizer $G$ containing $\tZeros(\tilde\phi_{d,e}(x))$. 
In particular $\cH^{\tId}=\tHilb_{p}^0$ and when $G = \tId$, we drop the superscript.
Note that we are making the implicit assumption that $G$ does not vary with the point $x$. 
\medskip

Set
$$
  \mpg \ = \ \cH^G/GL(T)
$$ 
and let 
$$
  \phi_{d,U} : U \to \mpg
$$
be the quotient map induced by $\hat\phi_{d,e}$. 
This extends to a well defined map
\be\label{gphidef}
\phi_d: X_{\tgeneral}\ra \mpg .
\ene
Let $\cS_d^{p,G}\subset \oplus_{\d=2}^dS^{\d}T^*$ denote
the inverse image $\hat\phi^{-1}_{d,e}(\cH^G)$ and let $\xpg$ and $\ppg$ denote the
images of the bundles $\bold\Xi_d$ and $\bold\Phi_d$ (Definition \ref{xiphidef}), restricted to $\cS_d^{p,G}$,
and pushed down to  $\mpg$.  Theorem \ref{phixthm} implies 
\begin{theorem}\label{phixthmb}
Let $X^n\subset\pp{n+1}$ be a hypersurface of degree at least $d$, with
nondegenerate Gauss map. Let $x\in X_{\tgeneral}$ and let $p$ denote
the Hilbert polynomial determined by the Fubini forms of degree $\le d$.
Let $G\subseteq GL(T)$ denote the stabilizer of $\tZeros(\phi_{d,e}(x))$ in the Hilbert scheme
$\tHilb_p$, which we assume to be independent of $x\in X_{\tgeneral}$, so the map \eqref{gphidef} is well defined.  
\begin{list}{\emph{(\alph{Lcnt})}}
{ \usecounter{Lcnt} \setlength{\leftmargin}{20pt}}
 \item Then $\gamma_{ \phi_{d }}(X_{\tgeneral})\subset \xpg$.
\item If $\tdeg(X)=d$, then 
$\gamma_{ \phi_{d }}(X_{\tgeneral})\subset \ppg$.
\end{list}
\end{theorem}

\begin{theorem} \label{thm:deg=d}
Let $X^n\subset\pp{n+1}$ be either smooth or such that $X$ is $\pp 1$-uniruled and for a general line $\ell\subset X$, $\ell\cap X_{\tsing}=\emptyset$.  If
for some $d\leq n$,  $\gamma_{ \phi_{d }}(X_{\tgeneral})\subset \ppg$, then $\tdeg(X)=d$.
\end{theorem}

\begin{proof}
If $X$ satisfies the hypotheses
of the theorem, then $F_{d+1}\subset I(\cC_{d,x})$ by Definition \ref{xiphidef}.  Therefore $\cC_{d,x}=\cC_{d+1,x}$. By Theorem 2 of \cite{L_LPV}, such $X$ are
$\pp 1$-uniruled.  But if $\tdeg(X)>d$, and $B$
is an irreducible variety of lines on $X$ that covers $X$, then
every line in $B$ intersects $X_{sing}$; see, e.g.,  \cite{LT_lines}.
\end{proof}

\noindent{\it Remark.}   
Compare Theorem  \ref{thm:deg=d} with Proposition 3.23 of \cite{L_CI}.  There, instead of imposing smoothness conditions on $X$, one requires that the generators of $I(\cC_{d,x})$ have no nontrivial syzygies, and concludes that $X$ has degree $d$ after $2d+1$ derivatives, as opposed to the $d+1$ derivatives of Theorem \ref{thm:deg=d}.

\medskip

Let $\pi=\pi_{p,G}: \cS^{p,G}_d \ra \cM^{p,G}_d$ denote the projection.
Set $\fm=\{ \bm = (m_{\t,\d}) \mid m_{\t,\d} \in S^{\t-\d}T^*\}$.  Then $\fm$ acts on $\oplus_{\d=2}^dS^{\d}T^*$ by
$$
  \bm.(P_2,\ldots, P_d) \ : = \ 
  \left( P_2 \, , \, P_3 + m_{3,2} P_2 \, , \, \ldots \, , \, 
  P_d + \textstyle\sum_{j=2}^{d-1} m_{d,j} P_j \right) \, .
$$
Additionally, the action of $\fgl(T)$ on $T$ induces an action on $S^\d T^*$ and thence on $\oplus_{\d=2}^dS^{\d}T^*$.  We denote the action by 
$X.(P_2 , \ldots , P_d) := (X.P_2 , \ldots , X.P_d)$.

\begin{proposition}\label{niceprop1}  With the notation as introduced above, we have
$$\tker \td\pi |_{\bold P}=(\fgl(T)+\fm).{\bold P}$$
\end{proposition}

\begin{proof} The fiber over $\pi(\bold P)$ of the projection $\cH^G\ra \cM^{p,G}_d$ is $PGL(T) \cdot I(\bold P)$. The fiber over $\pi_p(\bold P)=I(\bold P)$ of the projection from
$\cS_d^{p,G}\ra \cH^G$ is $M\cdot\bold P$, where $M$ is the Lie group associated to the Lie algebra $\fm$. 

Next, note that given $X\in \fgl(T)$ and $\bm\in \fm$ there exists
$X'\in \fgl(T)$, $\bm'\in \fm$ such that $X.\bm.\bold P=\bm'.X'.\bold P$.  It now follows that the fiber of the differential $\td \pi$ is the sum of the $\fgl(T)$ and $\fm$ actions. 
\end{proof}

\begin{proposition}\label{niceprop2} 
If $d\geq 4$ and $\bold P\in \oplus_{\d=2}^dS^{\d}T^*$ is sufficiently
general as described in the  proof, then for all $E\in \Xi_{\bold P}$,
\be\label{xieqn}
  E \cap \tker(\td\pi|_{\bold P})=0.
\ene
\end{proposition}

\begin{proof} 
Let $\bold P=(P_2\hd P_d)$ be such that $(P_2,P_3,P_4)$ is a generic triple. Any point of the linear space
$E$ is of the form  $(v\intprod P_3,v\intprod P_4\hd v\intprod P_d, v\intprod \a)$
for some $v\in T$ and $\a\in S^{d+1}T^*$.
A point of $\tker(\td\pi|_{\bold P})$ is of the form
$$
  \left( X.P_2 \, , \, X.P_3 + m_{3,2} P_2 \, , \, \ldots \, , \, 
  X.P_d + \textstyle\sum_{j=2}^d m_{d,j}P_j \right) \, .
$$
We can   find $X$ such that $X.P_2=v\intprod P_3$, but for generic triples $(P_2,P_3,P_4)$
the vector spaces $\{ v\intprod P_4\mid v\in T\}$, which is $n$-dimensional
and $X.P_3+m_{3,2}P_2$ which is less than $(n^2+n)$-dimensional will
have zero intersection in the $\binom{n+2}3$-dimensional $S^3T^*$.
\end{proof}

\noindent{\it Remark.} 
Note that assuming genericity of any of the $P_{\d}$ for $\d\geq 4$ would have been enough to conclude
\eqref{xieqn}.

\medskip

Combining Propositions \ref{niceprop1} and \ref{niceprop2}, we obtain

\begin{theorem}\label{forcedmove} For $d\geq 4$,
there is a non-empty Zariski open subset   $A_d\subset \cM_d$   such that
any $X^n\subset \pp{n+1}$ of degree at least $d$ with $x\in X_{\tgeneral}$ such that
$\phi_d(x)\in A_d$, must have $\trank(\td\phi_d|_x)=n$.
\end{theorem}
 
\medskip

\noindent{\it Remark.}  More generally, we have 
\begin{equation}\label{eqn:rank}
  \trank \td\tilde\phi_{d,e}(x) \ - \ \trank \td\phi_d(x) \ = \ 
  \tdim \left( \td\tilde\phi_{d,e}(x)(T) \, \cap \,  \left(
                \mathfrak{gl}(T) \, + \, \fm \right).\tilde\phi_{d,e}(x) \right) \, .
\end{equation}

\medskip

\noindent{\bf Example \ref{ranknexample} continued.}
Generically, the stabilizer of $p_d$ in $GL(T)$ will be trivial so that in 
\eqref{xieqn} we need only consider the intersection of $\langle p_{abc}\, y^b y^c \rangle_{a=1}^n$ with $\langle ( p_{ab} - p_{abc} z^c ) \, y^a y^b \rangle$.  Generically, the intersection will be trivial, so that the rank of $\td\phi_{3}$ is maximal. 
 
\medskip

\noindent{\it Remark.}  Had we normalized $F_2$ to be constant, $F_3$ would have been forced to vary.
See \S\ref{secondosect}.


\subsection{Sharper theorems}\label{sec:sharper}
Instead of descending all the way to $\mpg$, one may obtain similar results by descending just to $(\oplus_{\d=2}^dS^{\d}T^*)/H$ as we have a canonical map
$\overline\phi_d: X_{\tgeneral}\ra (\oplus_{\d=2}^dS^{\d}T^*)/H$. Here $H$ is the unipotent group preserving the flag $\hat x\subset \hat T_xX\subset V$.

\subsection{Second order adapted frames}\label{secondosect}

One could choose to work with second order adapted frames where $F_{2,X}$ is normalized to be a fixed quadratic form $Q$  (e.g. $Q=y_1{}^2+\cdots +y_n{}^2$), and  consider the map $\tilde\phi_{d,e}'$ from such a framing 
to $\oplus_{\d=3}^{d}S^{\d}T^*$.  In this case the derivative $d\tilde\phi_{d,e}'(z)$ is \emph{shifted}  in the following sense.  Contract   $P_{\d}\ot P_3$ via (the dual
quadric to) $Q$ to obtain an element of $S^{\d-1}T^*\ot S^2T^*$;
then symmetrize to get an element $\xi_{\d}(\bold P)\in S^{\d+1}T^*$. 
We find that
\begin{eqnarray*}
  \td \tilde \phi_{d,e}'(z) & = & (0,  \td F_{3,e}(z) , \ldots , \td F_{d,e}(z) ) \\
  & = & - ( F_{3,e}(z)+\xi_3(\bold P), \ldots , F_{d+1,e}(z) +\xi_d(\bold P)) \, .
\end{eqnarray*}

\section{Cubic examples}\label{sec:examples}

\subsection{Fermat cubic}
Given linear coordinates $\overline z = (z^0 , \ldots , z^\N ) \in \bC^{\N+1}$, consider the Fermat cubic
$$
  F(\overline z) \ = \ (z^0)^3 + \cdots + (z^\N)^3 \, .
$$
It is easy to see that $x = [ \N^{1/3} : -1 : \cdots : -1 ]$ is a smooth point of the hypersurface  $X = \{ F = 0 \}$.  In the cases $3 \le n \le 6$ we used Maple to confirm that the Fubini quadric $F_2$ is nondegenerate at $x$, and that $\td \phi_3$ has maximal rank at $x$.

\subsection{The determinant}\label{sec:det}
In contrast to the hypersurfaces discussed in this article, if $X$ is quasi-homogeneous, i.e.,
the Zariski closure in $\BP V$ of an orbit of a group $G$ acting linearly on $V$,
then its differential invariants will be constant on a Zariski open subset. More generally,
if it is a $G$-variety for some group $G$, then its differential invariants will
be constant along $G$-orbits. For example, consider the 
the cubic hypersurface $\tDet(3) \subset \bP^8$.  Given coordinates
$$
  w \ = \ \left( \begin{array}{ccc}
     w^0 & w^3 & w^8 \\
    w^6 & w^1 & w^4 \\
    w^5 & w^7 & w^2 
  \end{array} \right)  \ \in \ \bC^9 \, ,
$$
$\tDet(3)$ is given by the equation
$$
  F(w) \ = \ w^0 w^1 w^2 \ + \ w^3 w^4 w^5 \ + \ w^6 w^7 w^8 \ 
         - \ w^0 w^4 w^7 \ - \ w^2 w^3 w^6 \ - \ w^1 w^5 w^8  \, . 
$$
We will consider an open coordinate neighborhood of the smooth point $x = [1:1:0:\cdots:0] \in \tDet(3)$.  

Note that $\tDet(3)$ is preserved by the action of $GL(9)$ on $\bP^8$ and that $x$, corresponding to a rank two matrix, lies in the maximal orbit.  Hence we expect $\phi_3$ to be constant.
Indeed it is possible to construct a local, first-order adapted framing $e'$ in a neighborhood of $x$ with respect to which the Fubini invariants and thus
$\tilde\phi_{3,e'}$ are constant.  

However, the framing $e'$ is not of the form constructed in \S\ref{sec:coords}.
  Given a framing $e$ of the type constructed \S\ref{sec:coords}, $\td\tilde\phi_{3,e}$ will have maximal rank 7.  However,
$\td\tilde\phi_{3,e}(x)(T) \subset I^d(x) + \mathfrak{ gl}(T).\tilde\phi_{3,e}(x)$, and \eqref{eqn:rank} 
yields $\trank d\phi_{3,e}(x) = 0$.

The first three differential invariants in such a framing are
\begin{eqnarray*}
  F_{2,e}(x) & = & 2 \left( y^3 y^6 + y^5 y^7 \right) \, , \\
  F_{3,e}(x) & = & 6 \left( y^1 y^3 y^6 - y^1 y^5 y^7 
                   + y^2 y^5 y^6 + y^3 y^4 y^7 \right) \, , \\
  F_{4,e}(x) & = & 12 \left( (y^1)^2 + y^2 y^4 \right) 
               \, F_{2,e}(z) \, . 
\end{eqnarray*}
Note also that the quadric $F_{2,e}$ is singular.  This requires that we alter the moduli space that $\phi_3$ maps into. 

\subsection{The permanent}

Maintaining the coordinates of \S\ref{sec:det}, the equation of the permanent $\tPrm(3)$ is
$$
  F(w) \ = \ w^0 w^1 w^2 \ + \ w^3 w^4 w^5 \ + \ w^6 w^7 w^8 \ 
         + \ w^0 w^4 w^7 \ + \ w^2 w^3 w^6 \ + \ w^1 w^5 w^8  \, . 
$$
While the permanent is not invariant under the action of $GL(9)$ on $\bP^8$, it is invariant under left and right multiplication by diagonal matrices (with a one dimensional stabilizer) and a permutation group.  So we must have $\trank \td\phi_{3} \leq 8-(3+3-1)=3$.

If the entries $w^5, w^6, w^0, w^3, w^8$ in the first column and row of the matrix $w$ are nonzero, then they can be normalized to 1 by the group action.  We selected eight normalized points on the permanent, and found that that computations with \eqref{eqn:rank} (aided by Maple) yield $\trank \td\phi_3(x) = 3$ in each case.  Thus $\trank \td\phi_3(x) = 3$ at a general point $x$ on the permanent.

\bibliography{AG}
\bibliographystyle{alpha}
\end{document}